\documentclass[a4paper]{amsart}

\usepackage{amsmath}
\usepackage{amsthm}
\usepackage{amssymb}
\usepackage{xspace}

\theoremstyle{plain}
\newtheorem{theorem}{Theorem}
\newtheorem{lemma}[theorem]{Lemma}

\newcommand{\comment}[1]{}

\newcommand{\nat}{\ensuremath{\mathbf{N}}\xspace}
\newcommand{\intd}[1]{\,\mathrm{d}#1 \,}
\newcommand{\floor}[1]{\left\lfloor #1 \right\rfloor}
\newcommand{\hmeas}{\ensuremath{\mathcal{H}}\xspace}
\renewcommand{\P}{\ensuremath{\mathbf{P}}\xspace}
\newcommand{\oball}[2]{\ensuremath{B{\left(#1, #2\right)}}\xspace}

\DeclareMathOperator{\diam}{diam}

\DeclareMathOperator{\dimh}{dim_H}
\DeclareMathOperator*{\E}{\mathbf{E}}

\title{Hausdorff dimension of random limsup sets of balls in the unit cube}
\author{Fredrik Ekstr\"om}

\begin{document}
\maketitle

\noindent
An expanded version of this preprint was published in the American Mathematical
Monthly~\cite{Ekstrom2019}.

Fix a sequence $(r_n)$ of positive numbers and let $\omega = (\omega_n)$
be a sequence of points in $[0, 1]^d$, chosen randomly and independently
according to the uniform probability measure. Let
	\[
	E = \limsup_n \oball{\omega_n}{r_n};
	\]
this is the set of points that are covered infinitely many times
by the balls $\{\oball{\omega_n}{r_n}\}$. The purpose of this note
is to give a new simple proof of the following theorem
about the Hausdorff dimension of $E$. Let
	\[
	t_0 = \inf\left\{t; \, \sum_n r_n^t < \infty \right\}
	= \sup\left\{t; \, \sum_n r_n^t = \infty \right\}.
	\]

\begin{theorem} \label{thm1}
Almost surely, $\dimh E = \min\left(t_0, d\right)$.
\end{theorem}

This was proved for $d = 1$ by Fan and Wu~\cite{FanWu2004}, by constructing a
Cantor set inside $E$ and estimating its dimension. Durand~\cite{Durand2010}
gave another proof that works for arbitrary $d$, although he only
stated the theorem for $d = 1$. The argument uses ``homogeneous ubiquity''
to show that almost surely $E$ lies in a class of sets introduced by
Falconer~\cite{Falconer1994}, having the property that every countable intersection
of sets from the class has large Hausdorff dimension. A similar
argument
(see \cite[Proposition~4.7]{JarvenpaaJarvenpaaKoivusaloLiSuomala2014}
and ~\cite[Theorem~2.1]{JarvenpaaJarvenpaaKoivusaloLiSuomalaXiao2017})
using the mass transference principle~\cite{BeresnevichVelani2006}
instead of homogeneous ubiquity proves the analogue of Theorem~\ref{thm1}
in general Ahlfors regular spaces\,---\,Theorem~\ref{thm1} and
Theorem~\ref{thm2} below are special cases of this general result.
Using large intersection classes, Persson~\cite{Persson2015} proved a lower bound
for the limsup set of a sequence of randomly translated open sets in the
$d$-dimensional torus, which again implies Theorem~\ref{thm1}.
There are several other generalisations and variations of the theorem,
some of which can be found
in~\cite{EkstromJarvenpaaJarvenpaaSuomala2018, EkstromPersson2018,%
FengJarvenpaaJarvenpaaSuomala2018, Seuret2018, Zhang2012}.

The upper bound for $\dimh E$ is easily proved, using that every tail
of the sequence of balls covers $E$. The lower bound is easy
as well in case the sum in the definition of $t_0$ diverges for $t = d$:
By Borel--Cantelli's lemma any given point is almost
surely covered infinitely many times, and it follows
via Fubini's theorem that $E$ almost surely has full measure in
$[0, 1]^d$, hence dimension $d$. The present proof is based on a reduction
to a similar situation also when the sum diverges for $t < d$,
by viewing $[0, 1]^d$ as a product of a $t$-dimensional and a
$(d - t)$-dimensional space.

The letter $E$ will be used to denote a random limsup set of a sequence
of balls of radii $(r_n)$ also in a general metric probability space.
Of particular interest is the space $\Sigma = \{ 1, \ldots, b\}^\nat$,
consisting of the infinite strings on the symbols $1, \ldots, b$. It has
a metric given by
	\[
	d\left(\sigma, \widetilde \sigma \right) =
	\gamma\rho^{\min\{n; \, \sigma_n \neq \widetilde\sigma_n\}},
	\]
where $\gamma$ is a positive constant and $\rho$ is a parameter in $(0, 1)$,
and a uniform probability measure given by
	\[
	\mu = \left( \frac{1}{b} \sum_{q = 1}^{b} \delta_q \right)^\nat.
	\]
Let $\alpha = \log b / (-\log \rho)$. It is not too difficult to see that 
$\Sigma$ is \emph{$\alpha$-regular}, meaning that there is a constant $c \geq 1$
such that for every $\sigma$ and every $r \in (0, \diam \Sigma]$,
	\[
	c^{-1} r^\alpha \leq \mu\left(\oball \sigma r\right) \leq c r^{\alpha}.
	\]
This implies that $\dimh \Sigma = \alpha$.

Theorem~\ref{thm1} is a consequence of its counterpart in $\Sigma$.

\begin{theorem} \label{thm2}
In the space $\Sigma$ described above, almost surely
$\dimh E = \min\left(t_0, \alpha \right)$.
\end{theorem}

\begin{proof}[Proof that Theorem~\ref{thm2} implies Theorem~\ref{thm1}]
Clearly $\dimh E \leq d$, and using a tail of the sequence of balls as a
cover shows that $\hmeas_\delta^t(E) \leq 2^t \sum_n r_n^t < \infty$
for every $t > t_0$ and $\delta > 0$, so that $\dimh E \leq t_0$.
For the opposite inequality, take $b = 2^d$ and $\gamma = 2\sqrt d$ and $\rho = 1/2$.
Then $\alpha = d$, and the unit cube can be modelled by $\Sigma$ as follows.

Every closed dyadic cube in $[0, 1]^d$ of sidelength $2^{-(n - 1)}$ is
the union of $2^d$ closed dyadic cubes of sidelength $2^{-n}$, which
can be enumerated arbitrarily from $1$ to $2^d$. To every $\sigma$
in $\Sigma$ there is then a corresponding sequence $(D_n(\sigma))$
of closed dyadic cubes, such that for every $n$, the cube $D_n(\sigma)$
has sidelength $2^{-n}$ and is the subcube of $D_{n - 1}(\sigma)$
specified by $\sigma_n$. Define the map $\pi: \Sigma \to [0, 1]^d$ by
$\pi(\sigma) = \bigcap_n D_n(\sigma)$.

By the choice of $\gamma$, the map $\pi$ is a contraction, and the
preimage under $\pi$ of a closed dyadic cube is a ball in $\Sigma$
of the same diameter. Since moreover every ball in the unit cube can be covered
by a bounded number of dyadic cubes of comparable diameter, it
follows that $\dimh A = \dimh \pi A$ for every subset $A$ of $\Sigma$.

Let $\nu$ be the uniform probability measure on $[0, 1]^d$ and
define the probability spaces $(\Omega, \P)$ and $(\Omega', \P')$ by
	\[
	\Omega = \Sigma^\nat,\quad
	\P = \mu^\nat,\qquad
	\Omega' = \left([0, 1]^d\right)^\nat,\quad
	\P' = \nu^\nat.
	\]
Then $\nu$ is the image of $\mu$ under $\pi$, and thus $\P'$ is the image
of $\P$ under the induced map $\pi: \Omega \to \Omega'$.
For every $\omega = (\omega_n)$ in $\Omega$,
	\[
	\pi \left(\limsup_n \oball{\omega_n}{r_n}\right)
	\subset
	\limsup_n \big(\pi \oball{\omega_n}{r_n} \big)
	\subset
	\limsup_n \oball{\pi \omega_n}{r_n},
	\]
using at the second step that $\pi$ is a contraction, and thus
	\[
	\dimh E(\omega) = \dimh \pi E(\omega) \leq \dimh E(\pi\omega).
	\]
It follows that
	\begin{align*}
	\P'\left\{
	\dimh E(\omega') \geq \min\left(t_0, d\right)
	\right\}
	&=
	\P\left\{
	\dimh E(\pi\omega) \geq \min\left(t_0, d\right)
	\right\} \\
	&\geq
	\P\left\{
	\dimh E(\omega) \geq \min\left(t_0, d\right)
	\right\} = 1,
	\end{align*}
assuming Theorem~\ref{thm2} in the last step.
\end{proof}

The proof of Theorem~\ref{thm1} still works if the unit cube is replaced
by another space that has the same structure as $\Sigma$,
for example the middle thirds Cantor set with the uniform probability measure.

\section*{Two lemmas}
For $i = 1, 2$, let $X_i$ be a compact metric
space with a probability measure $\mu_i$.  Let $X = X_1 \times X_2$ with the
product metric given by
	\[
	d\big((x_1, x_2), \, (\widetilde x_1, \widetilde x_2) \big)
	=
	\max\big(
	d_1(x_1, \widetilde x_1), \, d_2(x_2, \widetilde x_2)
	\big);
	\]
thus a ball in $X$ is the product of balls in $X_1$ and $X_2$ of
the same radius. Let $E$ be the random limsup set obtained by choosing
the centres of the balls according to $\mu_1 \times \mu_2$, and
if $x_1 \in X_1$ let $E^{x_1}$ be the fibre of $E$ above $x_1$, that is,
	\[
	E^{x_1} =
	\left\{ x_2 \in X_2; \, (x_1, x_2) \in E \right\}.
	\]

\begin{lemma} \label{ABlemma}
Assume that $\mu_2$ has full support in $X_2$, and let
	\[
	A =
	\Big\{
	x_1 \in X_1; \, \sum_n \mu_1\big(\oball{x_1}{r_n} \big) = \infty
	\Big\}
	\]
and
	\[
	B(\omega) =
	\Big\{
	x_1 \in X_1; \, E^{x_1}(\omega) \text{ is residual}\,
	\Big\}.
	\]
Then almost surely, $\mu_1\left(A \setminus B \right) = 0$.

\begin{proof}
Fix $x_1$ in $A$. Let $V$ be an open subset of $X_2$ and let
$G_n$ be the event that $x_1 \in \oball{\omega_{n, 1}}{r_n}$ and
$\omega_{n, 2} \in V$. Then for every $n$,
	\[
	\P(G_n)
	=
	\P\left\{
	\omega_{n, 1} \in \oball{x_1}{r_n} \text{ and }
	\omega_{n, 2} \in V
	\right\}
	=
	\mu_1\big(\oball{x_1}{r_n}\big) \, \mu_2(V).
	\]
The sum over $n$ diverges since $\mu_2(V) > 0$, so by the Borel--Cantelli
Lemma there is almost surely some $n$ such that $G_n$ occurs. 
When $V$ ranges over a countable basis for the topology of
$X_2$ it follows that almost surely the set
$\{\omega_{n, 2}; \, x_1 \in \oball{\omega_{n, 1}}{r_{n, 1}} \}$
is dense in $X_2$, so that
	\[
	E^{x_1}(\omega) =
	\limsup\left\{ \oball{\omega_{n, 2}}{r_n}; \,
	x_1 \in \oball{\omega_{n, 1}}{r_n}
	\right\}
	\]
is residual. Fubini's theorem now gives
	\begin{align*}
	\E\big( \mu_1(A \setminus B) \big)
	&=
	(\P\times\mu_1) \left\{
	(\omega, x_1); \, x_1 \in A \setminus B(\omega)
	\right\} \\
	&=
	\int \P\left\{x_1 \in A \setminus B(\omega)\right\} \intd\mu(x_1) = 0,
	\end{align*}
which proves the lemma.
\end{proof}
\end{lemma}

If $I$ is a subset of $\nat$ let $\Sigma_I = \{1, \ldots, b\}^I$
with the metric
	\[
	d_I(\sigma, \widetilde \sigma) =
	\gamma\rho^{\min\{i \in I; \, \sigma_i \neq \widetilde \sigma_i \}}
	\]
and the probability measure
	\[
	\mu_I = \left( \frac{1}{b} \sum_{q = 1}^b \delta_q \right)^I.
	\]
Then $\Sigma = \Sigma_I \times \Sigma_J$, both as metric spaces and
probability spaces, where $J = \nat \setminus I$.

\begin{lemma} \label{regularlemma}
Given $t \in (0, \alpha]$, let
	\[
	I = \left\{\floor{\frac{\alpha k}{t}}; \, k \in \nat \right\}.
	\]
Then $\Sigma_I$ is $t$-regular, that is, there is a constant $c$
such that for every $\sigma \in \Sigma_I$ and every $r \in (0, \gamma\rho]$,
	\[
	c^{-1} r^t \leq \mu\left(\oball \sigma r \right) \leq c r^t.
	\]

\begin{proof}
Let $(i_k)$ be the strictly increasing enumeration of $I$ given by
$i_k = \floor{\frac{\alpha k}{t}}$. Take $\sigma \in \Sigma_I$ and
$r \in (0, \gamma\rho]$, and let $k$ be such that
	\[
	\gamma\rho^{i_{k + 1}} < r \leq \gamma\rho^{i_k}.
	\]
Then
	\[
	\oball\sigma r = \left\{
	\widetilde\sigma \in \Sigma_I; \,
	\widetilde\sigma_{i_1}\!\ldots\widetilde\sigma_{i_k}
	=
	\sigma_{i_1}\!\ldots\sigma_{i_k}
	\right\},
	\]
and hence
	\[
	\mu\left(\oball\sigma r\right) = b^{-k} = \rho^{\alpha k},
	\]
which is comparable to $r^t$.
\end{proof}
\end{lemma}

\section*{Proof of Theorem~\ref{thm2}}
The argument showing the the upper bound for $\dimh E$ in
the proof of Theorem~\ref{thm1} is valid also in $\Sigma$.

For the lower bound, let $t \leq \alpha$ be such that $\sum_n r_n^t = \infty$
and decompose $\Sigma$ as $\Sigma = \Sigma_I \times \Sigma_J$, where $I$ is as in
Lemma~\ref{regularlemma} and $J = \nat \setminus I$.
By Lemma~\ref{regularlemma} every ball in $\Sigma_I$ of radius $r$ has
measure at least $c^{-1}r^t$, and therefore the set $A$ in Lemma~\ref{ABlemma}
is the entire space $\Sigma_I$. Thus Lemma~\ref{ABlemma}
implies that almost surely the set $B$ has full $\mu_I$-measure, and
consequently $\dimh B \geq t$ since every ball in $\Sigma_I$ of radius $r$
has $\mu_I$-measure at most $c r^t$ by Lemma~\ref{regularlemma}. If
$\sigma_1 \in B$ then the fibre $E^{\sigma_1}$ is non-empty, and hence
$B \subset \pi_I E$ for every realisation of $B$ and $E$. Thus almost surely,
	\[
	\dimh E \geq \dimh \left(\pi_I E \right)
	\geq \dimh B \geq t.
	\]
Letting $t \to \min(t_0, \alpha)$ from below along a countable set
proves the almost sure lower bound for $\dimh E$.

\bibliographystyle{plain}
\bibliography{references}

\end{document}